\documentclass{article}
\usepackage{amssymb,amsmath,multirow}
\usepackage[dvips]{graphicx}
\usepackage{epsfig}
\usepackage{graphics}
\newtheorem{theorem}{Theorem}
\newtheorem{remark}{Remark}

\begin{document} 

\title{A Mixed-Binary Convex Quadratic Reformulation for Box-Constrained
Nonconvex Quadratic Integer Program
\thanks{This research was supported  by National Natural
Science Foundation of China under grant
91130019/A011702,
and by the fund of State Key Laboratory of
Software Development Environment under grant SKLSDE-2013ZX-13.}
}


\author{\small Yong Xia \\
{\small State Key Laboratory of Software Development
              Environment, LMIB of the Ministry of Education, }\\
{\small School of Mathematics and System Sciences, Beihang University, Beijing 100191, P. R. China}\\
{\small dearyxia@gmail.com }\\
\small  Ying-Wei Han \\
{\small School of Economics and Management, Beihang University, Beijing,
100191, P. R. China}\\
{\small ywwh@gmail.com} }

\date{Received: date / Accepted: date}

\maketitle

\begin{abstract}
In this paper, we propose a mixed-binary convex quadratic programming reformulation for the box-constrained nonconvex quadratic integer program
and then implement IBM ILOG CPLEX 12.6 to solve the new model.
Computational results demonstrate that our approach clearly outperform the very recent state-of-the-art solvers.
\end{abstract}
{\bf Keywords}: Quadratic Integer Programming \ Convex Quadratic Programming \  Semidefinite Programming

\section{Introduction.}  \label{sect1}
We study the box-constrained nonconvex quadratic integer program:
\begin{eqnarray*}
{\rm (P)}~&\min&  x^TQx+c^Tx\\
&{\rm s.t.}& l_i\le x_i\le u_i, ~i=1,\ldots,n,\\
&&x\in Z^n,
\end{eqnarray*}
where $Z^n$ denotes the set of $n$-dimensional vectors with integer entries, $l,u\in Z^n$ and $l_i<u_i$ for $i=1,\ldots,n$. Problem (P) includes
the binary quadratic program as a special case (i.e., $l_i=0$ and $u_i=1$ for $i=1,\ldots,n$) and hence it is in general NP-hard.

When $Q\succeq 0$ (i.e., $Q$ is positive semidefinite), the continuous relaxation is a convex program and provides an easy-to-solve lower bound of (P). In literature, there are
a few solvers to globally solve (P), see \cite{B08,B12}. Moreover, some softwares developed for integer nonlinear optimization can also globally solve (P), for example, IBM ILOG CPLEX 12.6. However,
CPLEX fails to solve (P) when $Q\not\succeq 0$. Actually, in this case, even the continuous relaxation of (P) remains NP-hard \cite{P91}. Very recently, in order to solve (P) with a nonconvex objective, two efficient branch-and-bound approaches based on semidefinite programming (SDP) relaxation \cite{B13b} and ellipsoidal relaxation \cite{B13a} are proposed, denoted by Q-MIST and GQIP, respectively. Computational results in \cite{B13a} show that they outperform existing solvers such as COUENNE \cite{B09} and BARON \cite{S10}.

Our contribution in this paper is to propose a mixed-binary convex quadratic programming reformulation of (P). Consequently, (P) now can be solved by IBM ILOG CPLEX  12.6.
The computational results demonstrate the high efficiency of our reformulation.

The paper is organized as follows. In Section 2, we present a mixed-binary convex quadratic programming reformulation of (P). In section 3, we report the    experimental
results. Conclusions are made in Section 4.

Throughout the paper, $v(\cdot)$ denotes the optimal value of
the problem $(\cdot)$. Notation $A\succeq 0$ implies that the matrix
$A$ is positive semidefinite.  For a
vector $a$, ${\rm Diag}(a)$ denotes
the diagonal matrix with $a$ being  its diagonal vector. $a\geq 0$ indicates that $a$ is componentwise nonnegative. For a real value $b$, $\lfloor b\rfloor$ is the largest integer less than or equal to $b$.

\section{A Mixed-Binary Convex Quadratic Reformulation}
In this section, we first present a family of mixed-binary quadratic programming (MBQP) reformulations of (P) and then identify the ``best'' MBQP which has the tightest continuous relaxation and the objective is convex.

Define
\begin{eqnarray*}
&&I=\left\{i\in\{1,\ldots,n\}:~\frac{u_i+l_i}{2}\in Z\right\},~J=\left\{i\in\{1,\ldots,n\}:~\frac{u_i+l_i}{2}\not\in Z\right\},\\
&&m_i=\left\{\begin{array}{ll}\frac{u_i-l_i}{2},&{\rm if}~i\in I\\
 u_i-l_i,& {\rm if}~i\in J\end{array}\right. ,~
\widetilde{x}_i=\left\{\begin{array}{ll}x_i-\frac{u_i+l_i}{2},&{\rm if}~i\in I,\\
2x_i-(u_i+l_i),& {\rm if}~i\in J.\end{array}\right.
\end{eqnarray*}
It is trivial to verify that (P) is equivalent to
\begin{eqnarray*}
{\rm (\widetilde{P})}~&\min&  \widetilde{x}^TQ\widetilde{x}+c^T\widetilde{x}\\
&{\rm s.t.}& -m_i\le \widetilde{x_i}\le m_i,~\widetilde{x}_i\in Z, ~\forall i\in I,\\
&& -m_i\le \widetilde{x_i}\le m_i,~\frac{\widetilde{x}_i-1}{2}\in Z, ~\forall i\in J.
\end{eqnarray*}

\begin{theorem} For any $\theta \in R^n$, ${\rm (\widetilde{P})}$ is equivalent to the following  linearly constrained
 mixed-binary quadratic program in the sense that both optimal objective values are equal and both share the same optimal solution $\widetilde{x}$:
\begin{eqnarray}
{\rm(MBQP_\theta)} &\min& \widetilde{x}^TQ\widetilde{x}+c^T\widetilde{x}+\sum_{i\in I} \theta_i\left(\widetilde{x}_i^2-\sum_{k=1}^{m_i} k^2 y_{ik}\right) \nonumber\\
&&~+\sum_{i\in J} \theta_i\left(\widetilde{x}_i^2-\sum_{k=0}^{(m_i-1)/2} (2k+1)^2y_{ik}
\right)\\
&{\rm s.t.}&
-\sum_{k=1}^{m_i} ky_{ik}\le \widetilde{x}_i\le
\sum_{k=1}^{m_i} ky_{ik},~\forall i\in I, \label{xy1}\\
&& z_i\le\sum_{k=1}^{m_i} y_{ik}\le 1,~\forall i\in I,\label{xy2}\\
&& \sum_{k=1}^{m_i} ky_{ik}-\widetilde{x}_i\le 2m_iz_i,~\forall i\in I, \label{z1}\\
&& \sum_{k=1}^{m_i} ky_{ik}+\widetilde{x}_i\le 2m_i(1-z_i),~\forall i\in I,\label{z2}\\
&&-\sum_{k=0}^{(m_i-1)/2} (2k+1)y_{ik}\le \widetilde{x}_i\le
\sum_{k=0}^{(m_i-1)/2} (2k+1)y_{ik},\forall i\in J,\label{j1}\\
&&  \sum_{k=0}^{(m_i-1)/2}y_{ik}=1,~\forall i\in J,\label{j2}\\
&& \sum_{k=0}^{(m_i-1)/2} (2k+1)y_{ik}-\widetilde{x}_i\le 2m_iz_i,~\forall i\in J,\label{z3}\\
&& \sum_{k=0}^{(m_i-1)/2} (2k+1)y_{ik}+\widetilde{x}_i\le 2m_i(1-z_i),~\forall i\in J,\label{z4}\\
&&y_{ik}\in\{0,1\},\forall i,k;~z \in \{0,1\}^n.\label{xy3}
\end{eqnarray}
\end{theorem}
\begin{proof}
Let $\widetilde{x}$ be any feasible solution of $(\widetilde{P})$.
For $i=1,\ldots,n$, 
we define
\[
\widetilde{k}=\left\{\begin{array}{ll}|\widetilde{x}_i|,&{\rm if}~ i\in I,\\\frac{|\widetilde{x}_i|-1}{2},&{\rm if}~ i\in J,\end{array}\right.
 \widetilde{z}_i= \left\{\begin{array}{ll}0,&{\rm if}~\widetilde{x}_i\ge 0,\\1,&{\rm otherwise.}\end{array}\right.
\]
Then we define $\widetilde{y}_{i\widetilde{k}}=1$ and $\widetilde{y}_{ik}=0$ for all $k\neq \widetilde{k}$. It is not difficult to verify that $(\widetilde{x},\widetilde{y},\widetilde{z})$ is a feasible solution of  ${\rm(MBQP_\theta)}$ with the same objective value as that of $(\widetilde{P})$ at $\widetilde{x}$.

Now, let
$(\widetilde{x},\widetilde{y},\widetilde{z})$ be a feasible solution of  ${\rm(MBQP_\theta)}$.

For any $i\in I$, the constraints (\ref{xy1}), (\ref{z1})-(\ref{z2}) and (\ref{xy3}) imply that
 either
 $\widetilde{x}_i=\sum_{k=1}^{m_i} k\widetilde{y}_{ik}$ or $\widetilde{x}_i=-\sum_{k=1}^{m_i} k\widetilde{y}_{ik}$.
Moreover, under the constraints (\ref{xy2}) and (\ref{xy3}), there is at most one element in $\{\widetilde{y}_{ik}:~k=1,\ldots,m_i\}$
equal to $1$ and the others are zeros. It follows that $\widetilde{x}_i\in\{-m_i,\ldots,m_i\}$ and $\widetilde{x}_i^2=\sum_{k=1}^{m_i} k^2\widetilde{y}_{ik}$.

For any $i\in J$, the constraints (\ref{j1}), (\ref{z3})-(\ref{z4}) and (\ref{xy3}) imply that
 either
 $\widetilde{x}_i=\sum_{k=0}^{(m_i-1)/2} (2k+1)\widetilde{y}_{ik}$ or $\widetilde{x}_i=-\sum_{k=0}^{(m_i-1)/2} (2k+1)\widetilde{y}_{ik}$.
Furthermore, under the constraints (\ref{j1}) and (\ref{j2}), there is exactly one element in $\{\widetilde{y}_{ik}:~k=0,\ldots,(m_i-1)/2\}$
equal to $1$ and the others are zeros. It follows that $\widetilde{x}_i\in\{-m_i,-m_i+2,\ldots,m_i-2,m_i\}$ and $\widetilde{x}_i^2=\sum_{k=0}^{(m_i-1)/2} (2k+1)^2\widetilde{y}_{ik}$.

Therefore, $\widetilde{x}$ is a feasible solution of $(\widetilde{P})$ with the same objective value as that of ${\rm(MBQP_\theta)}$ at $(\widetilde{x},\widetilde{y},\widetilde{z})$. The proof is complete.
\end{proof}

\begin{remark}
 We notice that ${\rm(MBQP_\theta)}$ is still equivalent to $(\widetilde{P})$ without the constraints in the left-hand side of (\ref{xy2}). However, these constraints help us to reduce many  feasible solutions such as $(x,y,z)$ with $y_{i1}=\ldots=y_{im_i}=0$ and $z_i=1$ for $i\in I$.
\end{remark}

\begin{remark}
A natural representation of $\{x_i:~ l_i\le x_i\le u_i,~x_i\in Z\}$ may be
\[
\left\{x_i=\sum_{k=l_i}^{u_i}ky_{ik}:~\sum_{k=l_i}^{u_i}y_{ik}=1,~ y_{ik}\in \{0,1\}^{u_i-l_i+1}\right\}.
\]
But it need $u_i-l_i+1$ additional binary variables $y_k$. In our reformulation ${\rm(MBQP_\theta)}$, we introduce an additional binary variable $z_i$ and $\frac{u_i-l_i}{2}$ or $\frac{u_i-l_i+1}{2}$ additional binary variables $y_{ik}$ depending on whether $i\in I$ or $i\in J$.
\end{remark}

The choice of $\theta$ plays a great role in solving ${\rm(MBQP_\theta)}$.
For any $\theta$ such that $Q+{\rm Diag}(\theta)\succeq 0$,  ${\rm(MBQP_\theta)}$ has a convex objective function and hence can be solved by IBM ILOG CPLEX. A ``best'' choice of $\theta$, denoted by $\theta^*$, seems to be the one that maximizes the continuous relaxation of ${\rm(MBQP_\theta)}$. For convenience, we rewrite the
continuous relaxation of ${\rm(MBQP_\theta)}$ as
\begin{eqnarray*}
{\rm R(\theta)}~&\min& \widetilde{x}^T(Q+{\rm Diag}(\theta))\widetilde{x}+c^T\widetilde{x}- L(\theta)^Ty\\
&{\rm s.t.}& A\widetilde{x}+By\le a,
\end{eqnarray*}
where $y=(y_{ik},z_i)$, $a$, $A, B$ are vectors/matrices of appropriate dimension and $L(\theta)$ is a linear operator of $\theta$.
Moreover, since $Q+{\rm Diag}(\theta)\succeq 0$, according to strong duality, we have
\begin{eqnarray}
&&v({\rm R(\theta)}) \nonumber\\
&=&\min_{x,y}\max_{\lambda\ge0} \left\{
\widetilde{x}^T(Q+{\rm Diag}(\theta))\widetilde{x}+c^T\widetilde{x}- L(\theta)^Ty-
 \lambda^T(a-A\widetilde{x}-By)\right\}\nonumber\\
&=&\max_{\lambda\ge0} \min_{x,y} \left\{
\widetilde{x}^T(Q+{\rm Diag}(\theta))\widetilde{x}+(c+A^T\lambda)^T\widetilde{x}- (L(\theta)-B^T\lambda)^Ty -a^T\lambda\right\}\nonumber\\
&=&\max_{\lambda\ge0}~~ \tau-a^T\lambda\nonumber\\
&&{\rm s.t.}~~ \left[\begin{array}{ccc}
-\tau& \frac{1}{2}(c+A^T\lambda)^T & \frac{1}{2}(-L(\theta)+B^T\lambda)^T \\
\frac{1}{2}(c+A^T\lambda) & Q+{\rm Diag}(\theta)& 0\\
\frac{1}{2}(-L(\theta)+B^T\lambda)&0 &0  \end{array} \right]\succeq 0.\label{cc}
\end{eqnarray}
Consequently,  $\theta^*$ is obtained by solving an SDP:
\begin{equation}
\theta^*={\rm arg}\max_{Q+{\rm Diag}(\theta)\succeq 0} v({\rm R(\theta)})
=\max_{\lambda\ge0, ~ (\ref{cc}) }~ \left\{\tau-a^T\lambda\right\},\label{sdp}
\end{equation}
where  $Q+{\rm Diag}(\theta)\succeq 0$ is removed since it is already implied from (\ref{cc}).

\section{Experimental Results}

In this section, we report computational results of ${\rm(MBQP_{\theta^*})}$, where $\theta^*$ is obtained by solving the SDP (\ref{sdp})
using SEDUMI \cite{S99} within  CVX 1.2 \cite{G10}.

We used the same test bed as in \cite{B13a}. We randomly generated two kinds of testing instances,
ternary instances where $l_i=-1$ and $u_i=1$ for all
$i = 1,\ldots, n$ and the instances with larger domain: $l_i=-5$ and $u_i=5$ for all $i = 1,\ldots, n$. For each $n\in \{20,30,40,50\}$ and each
$p\in\{0, 0.1, \ldots, 1\}$ (which is a parameter to control the inertia of $Q$),
we created $10$ random instances from different random seeds.
That is, we totally solved $110$ instances for each size $n$. For each instance,
we first
uniformly generated $\lfloor p\cdot n\rfloor$ and $n-\lfloor p\cdot n\rfloor$ random values from  $[-1, 0]$ and $[0, 1]$, respectively. Let these $n$ values be the eigenvalues of $Q$.  Together with a random orthogonal basis
of $R^n$, we created the Hessian matrix $Q$.
Finally, all entries of $c$ were  uniformly generated at random
from $[-1, 1]$.

All experiments were implemented in Matlab R2010a and IBM ILOG CPLEX 12.6.
We set the time limit to one hour and the tolerance of CPLEX to be $10^{-6}$
(note that the default tolerance is $10^{-4}$). As in \cite{B13a}, instances not solved to proven optimality (with accuracy less
than $10^{-6}$) within one hour are considered failures.
We compare our method with the very recent two efficient solvers, Q-MIST \cite{B13b} and GQIP \cite{B13a}. As their codes are
not available, we just take the time and nodes explored by Q-MIST and GQIP
from \cite{B13a} and showed in Table \ref{t1} the platform and software
that are used by \cite{B13a} and us, respectively. Notice that our machine is slower than they used in \cite{B13a}.

\begin{table}[t]\caption{Platform and software used for comparison.}\label{t1}
\begin{center}
\begin{tabular}{lll}
\hline
Platform & Q-MIST~and~GQIP & ${\rm MBQP_{\theta^*} }$\\
\hline
Processor  &Intel Core2 (3.2GHz) & Pentium Dual-Core (2.93GHz)  \\
Operating System  &Linux& Windows XP  \\
Memory Size &4 GB &  2 GB \\
Software Used &C++, Fortran 90 & Matlab R2010a, CPLEX 12.6 \\
Precision  &$10^{-6}$ &$10^{-6}$\\
\noalign{\smallskip}\hline
\end{tabular}
\end{center}
\end{table}

 \begin{table}[t]\caption{Comparative results for ternary instances.}\label{t2}
\begin{center}
\begin{tabular}{llllll}
\hline
n & ALG  & SOLVED & MAX~TIME & AVG~TIME & AVG~\#NODES\\
\hline
20 &{\rm Q-MIST} &110 &1.0 &0.2 &53.5 \\
   &GQIP &110 &0.6 &0.1 &1914.5 \\
   &${\rm MBQP_{\theta^*} }$ & 110 &1.4 &0.1 &82.6 \\
\hline
30 &{\rm Q-MIST} &110 &11.0& 2.0 &199.7\\
   &GQIP &110 &1.1 &0.5 &16435.7\\
   &${\rm MBQP_{\theta^*} }$ & 110 &4.2  &0.9  &453.0 \\
   \hline
40 &{\rm Q-MIST} &110 &105.0 &15.9 &831.6\\
   &GQIP &110 &10.0 &1.8 &307492.7\\
   &${\rm MBQP_{\theta^*} }$ & 110 &5.9 &  2.1 &2115.3 \\
   \hline
50 &{\rm Q-MIST} &110 &1573.0 &184.0 &5463.7\\
   &GQIP &110 &306.7 &31.6  &8614806.8\\
   &${\rm MBQP_{\theta^*}}$ & 110 &22.5 & 5.2  &11204.9 \\
\noalign{\smallskip}\hline
\end{tabular}
\end{center}
\end{table}

 \begin{table}[t]\caption{Comparative results for instances with variable domain $\{-5,\ldots, 5\}$.}\label{t3}
\begin{center}
\begin{tabular}{llllll}
\hline
n & ALG  & SOLVED & MAX~TIME & AVG~TIME & AVG~\#NODES\\
\hline
20 &{\rm Q-MIST} &110 &6.0 &0.8 &138.6 \\
   &GQIP &110 &56.9 &1.9 &1568914.8\\
   &${\rm MBQP_{\theta^*} }$ & 110 &0.7 & 0.2 &251.4 \\
\hline
30 &{\rm Q-MIST} &110 &237.0 &17.8 &1115.5\\
   &GQIP &103 &3175.1 &256.6 &151889362.6\\
   &${\rm MBQP_{\theta^*} }$ & 110 &12.9 & 1.0 &1840.0 \\
   \hline
40 &{\rm Q-MIST} &109 &2431.0 &211.3 &5861.5\\
   &GQIP &32 &3501.8 &600.6 &254751492.5\\
   &${\rm MBQP_{\theta^*} }$ & 110 &128.4 &   9.3 &15971.7 \\
\noalign{\smallskip}\hline
\end{tabular}
\end{center}
\end{table}

We compare the computational results for each dimension $n$ in Tables \ref{t2} and \ref{t3}, where the third column (SOLVED) gives the
number of instances solved to proven optimality within one
hour, out of all $110$ instances , the last three columns (MAX TIME, AVG TIME, AVG \#NODES) list  the maximum time, average time, and average number of nodes explored for the successfully solved instances, respectively.

According to Tables \ref{t2} and \ref{t3}, our method has the similar computational performance as Q-MIST and GQIP for the ternary instances with $n=20,30,40$, but highly outperforms Q-MIST and GQIP for the ternary instances
with $n=50$ and for the instances with the variable domain $\{-5,\ldots, 5\}$. We also see that the average numbers of nodes explored by CPLEX for solving ${\rm(MBQP_{\theta^*})}$ lie between those of Q-MIST and GQIP. These observations are reasonable as the continuous relaxations of the subproblems of ${\rm(MBQP_{\theta^*})}$ are
linearly constrained convex quadratic programs (CQP), which are weaker than those of Q-MIST (which are SDPs,  less efficient to solve than CQP), but generally tighter than those of GQIP (which are nonconvex quadratic programs with an ellipsoidal constraint).

\section{Conclusions}
 In this paper, we propose a family of mixed-binary quadratic programming reformulations for the box-constrained nonconvex quadratic integer program, denoted by ${\rm(MBQP_\theta)}$, where $\theta$ is a parameter vector. A ``best'' choice of $\theta$, denoted by $\theta^*$, is set as the one that maximizes the continuous relaxation of ${\rm(MBQP_\theta)}$. It turns out that $\theta^*$ can be obtained by solving a semidefinite program (SDP). Interestingly, ${\rm(MBQP_{\theta^*})}$ has a convex quadratic objective and hence can be solved by IBM ILOG CPLEX 12.6. Computational results demonstrate that for instances with large $n$ or large variable domain, ${\rm(MBQP_{\theta^*})}$ highly outperforms the recent efficient solvers,
 Q-MIST \cite{B13b} and GQIP \cite{B13a}, which are branch-and-bound approaches based on SDP relaxation and ellipsoidal relaxation, respectively.


\bibliographystyle{elsarticle-num}





\end{document}